\documentclass[12pt,reqno]{amsart}
\usepackage[left=80pt,right=80pt]{geometry}
\usepackage[usenames]{color}
\usepackage{amsmath}
\usepackage{amssymb}
\usepackage{amsthm}
\usepackage{float}
\usepackage{tikz}
\usepackage{mathdots}

\usepackage{hyperref,url}
\hypersetup{
	colorlinks=true,
  linkcolor=black,          
  citecolor=black,         
  filecolor=black,      
  urlcolor=black           
}

\newtheorem{prop}{Proposition}

\newtheorem{theorem}[prop]{Theorem}

\theoremstyle{definition}

\newtheorem{remark}[prop]{Remark}
\newtheorem{example}[prop]{Example}

\allowdisplaybreaks
\newcommand{\mylabel}[2]{#2\def\@currentlabel{#2}\label{#1}}
\makeatletter
\renewcommand*\env@matrix[1][\arraystretch]{%
  \edef\arraystretch{#1}%
  \hskip -\arraycolsep
  \let\@ifnextchar\new@ifnextchar
  \array{*\c@MaxMatrixCols c}}
\makeatother

\begin{document}
\tikzset{mystyle/.style={matrix of nodes,
        nodes in empty cells,
        row 1/.style={nodes={draw=none}},
        row sep=-\pgflinewidth,
        column sep=-\pgflinewidth,
        nodes={draw,minimum width=1cm,minimum height=1cmanchor=center}}}
\tikzset{mystyleb/.style={matrix of nodes,
        nodes in empty cells,
        row sep=-\pgflinewidth,
        column sep=-\pgflinewidth,
        nodes={draw,minimum width=1cm,minimum height=1cmanchor=center}}}

\title{A formula for the number of up-down words}

\author[SELA FRIED]{Sela Fried$^*$}
\thanks{$^*$Department of Computer Science, Israel Academic College,
52275 Ramat Gan, Israel.
\\
\href{mailto:friedsela@gmail.com}{\tt friedsela@gmail.com}}

\begin{abstract}
 \noindent
A word $w_1w_2\cdots w_n$ is said to be up-down if $w_1 < w_2 >w_3 \cdots$. Carlitz and Scoville found the generating function for the number of up-down words over an alphabet of size $k$. Using properties of the Chebyshev polynomials we derive a closed-form formula for these numbers.
\bigskip
 
\smallskip

\noindent
\textbf{Math.~Subj.~Class.:} 68R05, 05A05, 05A15.
\end{abstract}

\maketitle

\baselineskip=0.20in

\section{Introduction}
A permutation $\sigma=\sigma_1\sigma_2\cdots\sigma_n$ is said to be alternating  if $\sigma_1>\sigma_2<\sigma_3 \cdots$. By a classical result of Andr\'{e} \cite{andre1881permutations}, if $c_n$ stands for the number of alternating permutations of length $n$, then \[\sum_{n\geq 0}c_n\frac{x^n}{n!} = \sec{x}+\tan{x}.\] The analogue question regarding words seems to have been studied first by Carlitz and Scoville \cite{carlitz1972up}. Let $k\geq 1$ and $n\geq 0$ be two integers and set $[k] = \{1,2,\ldots,k\}$. A word over $k$ of length $n$ is an element of the set $[k]^n$. A word $w_1w_2\cdots w_n\in[k]^n$ is said to be up-down if $w_1 <w_2 >w_3 \cdots$. Let $f_{k,n}$ stand for the number of up-down words over $k$ of length $n$. In \cite{carlitz1972up} the generating functions for the number of up-down words over $k$ of even and odd length were found. Gao et al.\ \cite{gao2016pattern} proved via a bijection that the number of up-down words over $k$ of length $n$ is equal to the number of order ideals of certain posets. The enumeration problem of these order ideals was addressed by Berman and K\"{o}hler \cite[Example 2.3]{berman1976cardinalities}, who obtained a recursive formula for them. An evidently equivalent formulation of up-down words is that of a light ray entering a stack of glass plates such that each glass plate may reflect or transmit light. These were studied by \cite{junge1973polynomials}.

Despite these three different formulations of the same enumeration problem, it seems that the possibility of obtaining a closed formula was overlooked. The purpose of this work is to fill this gap. We show that 
\[f_{k,n}=1_{n=1}+\frac{4}{2k-1}\left|\sum_{i=1}^{k-1}\frac{\sin^{2}\left(\frac{2i\pi}{2k-1}\right)}{\left(2\cos\left(\frac{2i\pi}{2k-1}\right)\right)^{n+1}}\right|,
\] where $1_{n=1}$ equals $1$ if $n=1$ and $0$ otherwise.

\section{Main results}

Let 
\[F_k(x)=\sum_{n\geq 0}f_{k,2n+1}x^{2n+1} \;\;\;\textnormal{ and }\;\;\; G_k(x)=\sum_{n\geq 0}f_{k,2n}x^{2n}\] be the generating functions for the number of up-down words over $k$ of odd and even length, respectively. 
In \cite{carlitz1972up} it was shown that
\[F_k(x)=\frac{P_k(x)}{Q_k(x)} \;\;\;\textnormal{ and }\;\;\; G_k(x)=\frac{1}{Q_k(x)},\] where 
\begin{align}
P_k(x)&=\sum_{i=0}^k(-1)^i\binom{k+i}{2i+1}x^{2i+1},\nonumber\\
Q_k(x)&=\sum_{i=0}^k(-1)^i\binom{k+i-1}{2i}x^{2i}.\nonumber
\end{align}

The key insight in translating these generating functions into a closed formula is their expressibility in terms of the Chebyshev polynomials. These polynomials are related to the cosine and sine functions and find use in approximation theory (e.g., \cite{mason2002chebyshev}). They also emerge naturally in combinatorics (e.g., \cite{knopfmacher2010staircase}). The Chebyshev polynomials of the first kind, denoted by
$T_n(x)$, are defined by 
$T_{n}(\cos \theta )=\cos(n\theta)$ and the Chebyshev polynomials of the second kind, denoted by $U_n(x)$, are defined by 
$U_{n}(\cos\theta )\sin\theta=\sin((n+1)\theta)$. Finally, the Chebyshev polynomials of the third kind, denoted by $V_n(x)$, are defined by 
$V_{n}(\cos\theta )\cos\left(\frac{1}{2}\theta\right)=\cos\left(\left(n+\frac{1}{2}\right)\theta\right)$. 
With the help of these polynomials we prove the following result.
\begin{theorem}\label{c1}
The number of up-down words over $k$ of length $n$ is given by
\begin{equation}\label{e0}
f_{k,n}=1_{n=1}+\frac{4}{2k-1}\left|\sum_{i=1}^{k-1}\frac{\sin^{2}\left(\frac{2i\pi}{2k-1}\right)}{\left(2\cos\left(\frac{2i\pi}{2k-1}\right)\right)^{n+1}}\right|.
\end{equation} where $1_{n=1}$ equals $1$ if $n=1$ and $0$ otherwise.
\end{theorem}

\begin{proof}
It is well known that 
\begin{align}
T_{m}(x)&=m\sum_{i=0}^{m}(-2)^i\frac{(m+i-1)!}{(m-i)!(2i)!}(1-x)^i,\nonumber\\
U_{m}(x)&=\sum_{i=0}^{m}(-2)^{i}\binom{m+i+1}{2i+1}(1-x)^{i}.\nonumber
\end{align} It immediately follows that \[P_{k}(x)=xU_{k-1}\left(1-\frac{x^{2}}{2}\right).\] We claim that \[Q_{k}(x)=V_{k-1}\left(1-\frac{x^{2}}{2}\right).\] Indeed, using 
\begin{align}
T_m(x)&=\frac{1}{2}(U_m(x)-U_{m-2}(x)),\nonumber\\
\frac{d}{dx}T_m(x)&=mU_{m-1}(x),\nonumber\\
U_m(x)&=2xU_{m-1}(x)-U_{m-2}(x),\nonumber\\
V_m(x)&=U_m(x)-U_{m-1}(x),\nonumber
\end{align} we have
\begin{align}
Q_{k}(x)&=\sum_{i=0}^{k-1}(-1)^{i}\binom{k+i-1}{2i}x^{2i}\nonumber\\
&=k\sum_{i=0}^{k-1}(-2)^{i}\frac{(k+i-1)!}{(k-i)!(2i)!}\left(1-\left(1-\frac{x^{2}}{2}\right)\right)^{i}\nonumber\\&-\sum_{i=0}^{k-1}(-2)^{i}i\frac{(k+i-1)!}{(k-i)!(2i)!}\left(1-\left(1-\frac{x^{2}}{2}\right)\right)^{i}\nonumber\\
&=T_k\left(1-\frac{x^{2}}{2}\right)-\frac{x}{2k}\frac{d}{dx}T_k\left(1-\frac{x^{2}}{2}\right)\nonumber\\
&=U_{k-1}\left(1-\frac{x^{2}}{2}\right)-U_{k-2}\left(1-\frac{x^{2}}{2}\right)\nonumber\\
&=V_{k-1}\left(1-\frac{x^{2}}{2}\right)\nonumber.
\end{align} Thus,
\[F_k(x)=x+\frac{xU_{k-2}\left(1-\frac{x^{2}}{2}\right)}{V_{k-1}\left(1-\frac{x^{2}}{2}\right)} \;\;\;\textnormal{ and }\;\;\; G_k(x)=\frac{1}{V_{k-1}\left(1-\frac{x^{2}}{2}\right)}.\]

We now proceed as was done in \cite{knopfmacher2010staircase}. It is well-known that $V_m(x)$ is a polynomial of degree $m$ with $m$ distinct roots given by \[\cos\left(\frac{\left(i-\frac{1}{2}\right)\pi}{m+\frac{1}{2}}\right), \;\;i=1,\ldots,m.\] Furthermore, the coefficient of $x^m$ in $V_m(x)$ is $2^m$. Hence,
\begin{align}
V_{k-1}\left(1-\frac{x^{2}}{2}\right)&=2^{k-1}\prod_{i=1}^{k-1}\left(1-\frac{x^2}{2}-\cos\left(\frac{\left(i-\frac{1}{2}\right)\pi}{k-\frac{1}{2}}\right)\right)\nonumber\\&=\prod_{i=1}^{k-1}\left(2\sin\left(\frac{\left(i-\frac{1}{2}\right)\pi}{2k-1}\right)-x\right)\left(2\sin\left(\frac{\left(i-\frac{1}{2}\right)\pi}{2k-1}\right)+x\right)\nonumber.
\end{align} Now, by partial fraction decomposition (e.g., \cite[(5.8)]{aramanovich2014mathematical}), if $q(x)$ is a polynomial of degree $m$ whose roots are all real and distinct, which factors as $q(x)=(x-\beta_1)\cdots(x-\beta_m)$, and $p(x)$ is a polynomial whose degree is strictly less than $m$, then \[\frac{p(x)}{q(x)}=\sum_{i=1}^{m}\frac{p(\beta_{i})}{q'(\beta_{i})}\frac{1}{x-\beta_{i}}.\] Since \[\frac{d}{dx}U_{m}(x)=\frac{(m+1)T_{m+1}(x)-xU_{m}(x)}{x^{2}-1},\] we have
\[\frac{d}{dx}V_{k-1}\left(1-\frac{x^{2}}{2}\right)=\frac{2x\left(kU_{k-2}\left(1-\frac{x^{2}}{2}\right)+(k-1)U_{k-1}\left(1-\frac{x^{2}}{2}\right)\right)}{x^{2}-4}.\] Thus, 
\[\frac{d}{dx}V_{k-1}\left(1-\frac{x^{2}}{2}\right)_{|x=\pm 2\sin\left(\frac{\left(i-\frac{1}{2}\right)\pi}{2k-1}\right)}=\pm \frac{(-1)^{i}(2k-1)}{2\cos^{2}\left(\frac{\left(i-\frac{1}{2}\right)\pi}{2k-1}\right)}.\] Furthermore, 
\[xU_{k-2}\left(1-\frac{x^{2}}{2}\right)_{|x=\pm 2\sin\left(\frac{\left(i-\frac{1}{2}\right)\pi}{2k-1}\right)}=\pm (-1)^{i+1}.\] We conclude that
\[
f_{k,n} =\frac{1}{2^{n-1}(2k-1)}\sum_{i=1}^{k-1}(-1)^{i+1}\frac{\cos^{2}\left(\frac{\left(i-\frac{1}{2}\right)\pi}{2k-1}\right)}{\sin^{n+1}\left(\frac{\left(i-\frac{1}{2}\right)\pi}{2k-1}\right)},
\] if $n$ is even, and
\[f_{k,n} = 1_{n=1}+\frac{1}{2^{n-1}(2k-1)}\sum_{i=1}^{k-1}\frac{\cos^{2}\left(\frac{\left(i-\frac{1}{2}\right)\pi}{2k-1}\right)}{\sin^{n+1}\left(\frac{\left(i-\frac{1}{2}\right)\pi}{2k-1}\right)},\] if $n$ is odd.
It follows that, for arbitrary $n$, 
\[
f_{k,n} =1_{n=1}+\frac{4}{2k-1}\sum_{i=1}^{k-1}\frac{\cos^{2}\left(\frac{\left(i-\frac{1}{2}\right)\pi}{2k-1}\right)}{\left(2(-1)^{i+1}\sin\left(\frac{\left(i-\frac{1}{2}\right)\pi}{2k-1}\right)\right)^{n+1}},\] from which formula \eqref{e0} easily follows.
\end{proof}

\begin{example}
For $k=2$, formula \eqref{e0} reduces to     
\[f_{2,n}=1_{n=1}+\frac{4}{3}\left|\frac{\sin^{2}\left(\frac{2\pi}{3}\right)}{\left(2\cos\left(\frac{2\pi}{3}\right)\right)^{n+1}}\right|=1_{n=1}+\frac{1}{3\cdot 2^{n-1}}\left|\frac{\left(\frac{\sqrt{3}}{2}\right)^{2}}{\left(-\frac{1}{2}\right)^{n+1}}\right|=1_{n=1}+1,\] as it should, since, in this cases, there is only one up-down word for each $n\neq 1$, namely, $121\cdots$, and, for $n=1$, there are two such words, namely $1$ and $2$.

For $k=3$, it was shown in  \cite[Theorem 2.3]{gao2016pattern} that $f_{3,n} = F_{n+2}$, for $n\geq 2$, where $F_m$ stands for the $m$th Fibonacci number. This, of course, follows also from formula \eqref{e0}. Indeed,
\begin{align}
f_{3,n}&=\frac{4}{5}\left|\frac{\sin^{2}\left(\frac{2\pi}{5}\right)}{\left(2\cos^{n+1}\left(\frac{2\pi}{5}\right)\right)^{n+1}}+\frac{\sin^{2}\left(\frac{4\pi}{5}\right)}{\left(2\cos^{n+1}\left(\frac{4\pi}{5}\right)\right)^{n+1}}\right|\nonumber\\
&=\frac{1}{2}\left(\left(1+\frac{1}{\sqrt{5}}\right)\left(\frac{1+\sqrt{5}}{2}\right)^{n+1}+\left(1-\frac{1}{\sqrt{5}}\right)\left(\frac{1-\sqrt{5}}{2}\right)^{n+1}\right)\nonumber\\
&=\frac{1}{2}(F_{n+1}+L_{n+1})\nonumber\\
&=F_{n+2},\nonumber    
\end{align} where $L_m$ stands for the $m$th Lucas number (e.g., \cite[pp.~132-133]{stein1972john}).
\end{example}

\begin{remark}
Replacing $<$ and $>$ with $\leq$ and $\geq$, respectively, in the definition of up-down words, gives rise to weakly up-down words. These seem not to have been considered at all, but turn out to be essentially the same as up-down words, since, for $n\neq 1$, the number of weakly up-down words over $k$ of length $n$ is equal to the number of up-down words over $k+1$ of length $n$. Indeed, let $w_1\cdots w_n\in[k]^n$ be a weakly up-down. Define an up-down word $u_1\cdots u_n\in[k+1]^n$ as follows: 
\[u_{i}=\begin{cases}
w_{i}, & \text{if \ensuremath{i} is odd or }w_{i-1}<w_{i}>w_{i+1};\\
k+1, & \text{otherwise}.
\end{cases}\] Conversely, let $u_1\cdots u_n\in[k+1]^n$ be an up-down word and set $u_0=u_{n+1}=0$. Define a weakly up-down word $w_1\cdots w_n\in[k]^n$ as follows: 
\[w_{i}=\begin{cases}
u_{i}, & \text{if }u_{i}<k+1;\\
\max\left\{u_{i-1},u_{i+1}\right\} , & \text{otherwise}.
\end{cases}\] It is easy to see that the two maps are inverse to one another and hence bijections.
Thus, the number of weakly up-down words over $k$ of length $n$ is given by
\[\frac{4}{2k+1}\left|\sum_{i=1}^{k}\frac{\sin^{2}\left(\frac{2i\pi}{2k+1}\right)}{\left(2\cos\left(\frac{2i\pi}{2k+1}\right)\right)^{n+1}}\right|.
\] 
\end{remark}

We now consider the cyclic analogue of up-down words. A word $w_1\cdots w_n\in[k]^n$ of even length $n$ is said to be cyclic up-down if it is up-down and $w_n>w_1$.

\begin{theorem}
Let $A(x)=\sum_{n\geq 0} a_{k,2n}x^{2n}$ be the generating function for the number of cyclic up-down words of even length. Then 
\begin{equation}\label{k2}
A(x) = 1 +\frac{x^{2}}{4-x^{2}}\left(k-1+(2k-1)\frac{U_{k-2}\left(1-\frac{x^{2}}{2}\right)}{V_{k-1}\left(1-\frac{x^{2}}{2}\right)}\right).
\end{equation} In particular, 
\begin{equation}\label{k3}
a_{k,2n} = \sum_{i=1}^{k-1}\frac{1}{\left(2\cos\left(\frac{2i\pi}{2k-1}\right)\right)^{2n}}.
\end{equation}
\end{theorem}

\begin{proof}
Let us modify $F_m(x)$ and replace it with \[\frac{1}{x}-x+F_{m}(x)=\frac{1}{x}+\frac{xU_{m-2}\left(1-\frac{x^{2}}{2}\right)}{V_{m-1}\left(1-\frac{x^{2}}{2}\right)}.\] In particular, $f_{m,-1} = 1$ and $f_{m,1} = m-1$. We count the number of cyclic up-down words over $k+1$ according to the number of occurrences and position of $k+1$. Notice that the letter $k+1$ can be placed only at an even letter index. There are $h_{k,2n}$ cyclic up-down words without any $k+1$ and $nf_{k,2n-1}$ such words with exactly one occurrence of $k+1$. Now assume that there are at least two occurrences of $k+1$ and denote by $i$ and $j$ the leftmost and rightmost index of these $k+1$s, respectively. There are $f_{k+1,2(j-i)-1}$ possibilities for the subword $w_{i+1}\cdots w_{j-1}$ and $f_{k,2n-2(j-i)-1}$ possibilities for the (cyclic) subword $w_{j+1}\cdots w_{2n}w_1\cdots w_{i-1}$. It follows that 
\begin{align}
a_{k+1,2n}&=a_{k,2n}+nf_{k,2n-1}+\sum_{i=1}^{n-1}\sum_{j=i+1}^{n}f_{k+1,2(j-i)-1}f_{k,2n-2(j-i)-1}\nonumber\\
&=a_{k,2n}+\sum_{i=1}^{n}if_{k+1,2(n-i)-1}f_{k,2i-1}.\nonumber
\end{align} Multiplying both sides of the equation by $x^{2n}$ and summing over $n\geq 1$, we obtain \[A_{k+1}(x)=A_{k}(x)+\frac{x^{2}}{2}F_{k+1}(x)\left(2x+\frac{d}{dx}\left(xF_{k}(x)\right)\right).\] It follows that \[A_{k}(x)=1+\frac{x^{2}}{2}\sum_{i=1}^{k-1}F_{i+1}(x)\left(2x+\frac{d}{dx}\left(xF_{i}(x)\right)\right).\] One may show that this expression yields \eqref{k2}, from which the closed formula easily follows.
\end{proof}

\begin{example}
For $k=2$, formula \eqref{k3} should always give $1$. And, indeed, \[a_{2,2n} =\frac{1}{\left(2\cos\left(\frac{2\pi}{3}\right)\right)^{2n}}=\frac{1}{\left(-2\cdot\frac{1}{2}\right)^{2n}}=1.\] Consider now $k=3$. We have  
\begin{align}
a_{3,2n} &= \frac{1}{\left(2\cos\left(\frac{2\pi}{5}\right)\right)^{2n}}+\frac{1}{\left(2\cos\left(\frac{4\pi}{5}\right)\right)^{2n}}\nonumber\\
&=\frac{1}{\left(2\left(\frac{-1+\sqrt{5}}{4}\right)\right)^{2n}}+\frac{1}{\left(2\left(\frac{-1-\sqrt{5}}{4}\right)\right)^{2n}}\nonumber\\
&=\left(\frac{1+\sqrt{5}}{2}\right)^{2n}+\left(\frac{1-\sqrt{5}}{2}\right)^{2n}\nonumber\\&=L_{2n}    \nonumber.
\end{align} This seems to give a new combinatorial interpretation to the bisection of the Lucas numbers (cf.\ sequence A005248 in \cite{oeis}).
\end{example}

\bibliographystyle{vancouver}
\bibliography{Vancouver.bib}
\end{document}